\documentstyle[twoside, amsmath, amssymb, amsfonts]{article}

\newtheorem{theorem}{Theorem}[section]
\newtheorem{lemma}[theorem]{Lemma}
\newtheorem{remark}[theorem]{Remark}
\newtheorem{proposition}[theorem]{Proposition}
\newtheorem{corollary}[theorem]{Corollary}
\newtheorem{definition}[theorem]{Definition}
\newtheorem{example}[theorem]{Example}
\newenvironment{proof}{\begin{trivlist} \item[]{\em Proof.}}{\end{trivlist}}
\newcount\refno
\newcommand\be{\begin{equation}}
\newcommand\ee{\end{equation}}
\newcommand\bn{\begin{eqnarray}}
\newcommand\en{\end{eqnarray}}
\newcommand\bns{\begin{eqnarray*}}
\newcommand\ens{\end{eqnarray*}}
\newcommand\bd{\begin{definition}}
\newcommand\ed{\end{definition}}
\newcommand\br{\begin{remark}}
\newcommand\er{\end{remark}}
\newcommand\bt{\begin{theorem}}
\newcommand\et{\end{theorem}}
\newcommand\bp{\begin{proposition}}
\newcommand\ep{\end{proposition}}
\newcommand\bc{\begin{corollary}}
\newcommand\ec{\end{corollary}}
\newcommand\bl{\begin{lemma}}
\newcommand\el{\end{lemma}}
\newcommand\pf{\begin{proof}}
\newcommand\qed{\end{proof}\eop}

\newcommand\bN{{\mathbb N}}

\newcommand{\stirl}[2]{\genfrac{\{}{\}}{0pt}{}{#1}{#2}}
\def\eop{\hfill\rule{2.0mm}{2.0mm}}

\makeindex
\begin{document} 

\title{Divisibility of the Sums of the Power of Consecutive Integers}

\author{Tian-Xiao He$^{1}$ and Peter J.-S. Shiue$^2$ \\
{\small $^{1}$Department of Mathematics}\\
 {\small Illinois Wesleyan University, Bloomington, Illinois 61702, USA}\\
{\small $^2$ Department of Mathematical Sciences}\\
{\small University of Nevada, Las Vegas}\\
{\small Las Vegas, Nevada,  89154-4020, USA}\\
}

\date{\small Dedicated to Professor Anthony G. Shannon on the occasion of his 85th birthday}


\maketitle
\setcounter{page}{1}
\pagestyle{myheadings}
\markboth{T.-X.  He, and P. J.-S. Shiue}
{Divisibility of the Sums of the Power of Consecutive Integers } 

\begin{abstract}
\noindent 
We study the divisibility of the sums of the odd power of consecutive integers, $S(m,k)=1^{mk}+2^{mk}+\cdots+k^{mk}$ and $1^k+2^k+\cdots+n^k$ for odd integers $m$ and $k$, by using the Girard-Waring identity. Faulhaber's approach for the divisibilities is discussed. Some expressions of power sums in terms of Stirling numbers of the second kind are represented.


\end{abstract}

\section{Introduction}

Sums of powers of integers $1^n+2^n+\cdots+m^n$ have been studied for hundreds of years and even now there is still a steady stream of articles published on the subject. Jacques Bernoulli (\cite{Ber}, pp. 95--97) had introduced the numbers called after his name to evaluate the sum of the $n$-th powers of the first $m$ integers. He then proved the following summation formula (see, for example, (6.78) of Graham, Knuth, and Patashnik \cite{GKP} and Theorem 1.19 of Kerr \cite{Kerr}) .

\bn\label{0.0}
&&\sum^n_{\ell=1}\ell^k=\frac{1}{k+1}\sum^k_{\ell =0}(-1)^\ell B_\ell \binom{k+1}{\ell }n^{k-\ell +1}\nonumber\\
&=&\frac{1}{k+1}\left(B_{k+1}(n+1)-B_{k+1}(1)\right),
\en
where $n, k\geq 1$ and $B_\ell $ are the Bernoulli numbers with $B_1=-1/2$. 

The Faulhaber theorem states that the sum of the odd powers  
\[
1^{2m-1}+2^{2m-1}+\cdots+n^{2m-1}
\]
can be expressed as a polynomial of the triangular number $T_n=n(n+1)/2$. Let $\lambda =m(m+2x+1)$ be the sum of $x+1,x+2,\ldots, x+m$. Then extended and generalized results of the sums of integer powers in terms of the sum of arithmetic series, $\lambda$, are surveyed and developed in Chen, Fu, and Zhang \cite{CFZ}, which are given as 

\be\label{0.0-2}
\sum^n _{\ell =1}(x+\ell)^{2k -1}=\sum^k _{\ell =1}\frac{\lambda^\ell }{2k }\sum^k _{j =\ell }\binom{2k }{2j }\binom{j }{\ell }\left( x+\frac{1}{2}\right)^{2j -2\ell }B_{2k -2j }\left(\frac{1}{2}\right).
\ee

In this paper, we will study the divisibility of the sums of powers by using the Girard-Waring identity of the sums of two power terms, $x^n+y^n$. 
Albert Girard published a class of identities in Amsterdam in 1629 and Edward Waring published similar material in Cambridge in 1762-1782, which are referred as the Girard-Waring identities later. These identities may be derived from the earlier work of Sir Isaac Newton. Surveys and some applications of these identities can be found in  Comtet \cite{Com} (p. 198), Gould \cite{Gou}, Shapiro and one of the authors \cite{HS16}, and \cite{HS20}. Nie, Chen and the author's \cite{HSNC} give a different approaches to derive the Girard-Waring identities by using the Binet formula of recursive sequences and divided differences. Meanwhile, this approach offers some formulas and identities that have more wider applications, for instance, the transformations of certain recursive sequences to the Chebyshev polynomials of the first kind and the Chebyshev polynomials of the second kind shown in Weng and the authors \cite{HSW}. Recently, \cite{HS21} present a general rule of construction of identities for recursive sequences by using sequence transformation techniques developed in \cite{HSNC}. In this sense, this paper is a derivative, successor, and development of \cite{HSNC} and \cite{HS21}.

The Girard-Waring identity and it's Binet form can be presented by 
\bn
&&x^n + y^n = \sum_{0 \leq k \leq \left[n/2\right]}(-1)^k \frac{n}{n-k}\binom{n-k}{k}(xy)^k(x + y)^{n-2k},\label{eq:1.1}\\
&&\frac{x^{n+1}-y^{n+1}}{x-y}=\sum_{0 \leq k \leq \left[n/2\right]}(-1)^k \binom{n-k}{k}(xy)^k(x + y)^{n-2k}.\label{eq:1.2}
\en
Shapiro and one of the authors \cite{HS16} used Riordan array approach to establish the Binet type Girard-Waring identity \eqref{eq:1.2}. \cite{HSNC} establishes the formula 
\bns
&&\left( \frac{a_1-y  a_0}{x -y }\right) x ^n- \left(\frac{a_1-x  a_0}{x -y }\right) y ^n=a_{1}(x +y )^{n-1}\\
&&\quad +\sum^{[n/2]}_{j=1}\frac{1}{j}\binom{n-j-1}{j-1}
(-1)^j(x  +y )^{n-2j-1}(x y )^{j}\left( j(x +y )a_{0}+(n-2j)a_{1}\right),
\ens
which includes Binet Girard-Waring identity as a special case of $a_0=0$ and $a_1=1$.  

There are some alternative forms of the Girard-Waring identity. As an example, we present the following one. If $x+y+z=0$, then the Giraid-Waring identity gives 

\bns
&&x^n+y^n=\sum_{0 \leq k \leq [n / 2]}\qquad (-1)^k \frac{n}{n-k} \binom{n-k}{k} (-z)^{n-2k}(xy)^k\\
&=& (-1)^nz^n+\sum_{1 \leq k \leq [n / 2]}\qquad (-1)^{n-k} \frac{n}{n-k} \binom{n-k}{k} z^{n-2k}(xy)^k,
\ens
which implies 

\[
x^n+y^n-(-1)^nz^n=\sum_{1 \leq k \leq [n / 2]}\qquad (-1)^{n-k} \frac{n}{n-k} \binom{n-k}{k}z^{n-2k}(xy)^k.
\]
Thus, when $n$ is even, we have formula 

\be\label{1.1-2}
x^n+y^n-z^n=\sum_{1 \leq k \leq n / 2}\qquad (-1)^{n-k} \frac{n}{n-k} \binom{n-k}{k}z^{n-2k}(xy)^k,
\ee
while for odd $n$ we have 

\be\label{1.1-3}
x^n+y^n+z^n=\sum_{1 \leq k \leq [n / 2]}\qquad (-1)^{n-k} \frac{n}{n-k} \binom{n-k}{k}z^{n-2k}(xy)^k,
\ee
where $x+y+z=0$. Particularly, if $n=3$, then 

\be\label{1.1-4}
x^3+y^3+z^3=3xyz,
\ee
which was shown in \cite{HS20}. Furthermore, the following proposition was represented in \cite{HS20}, which can also be observed directly from \eqref{1.1-3}.

\begin{proposition}\label{pro:3.3}
Let $x,y\in{\bN}$. Then $pxy(x+y)|(x^p+y^p-(x+y)^p)$ when $p \geq 3 $ is a prime. 
\end{proposition}

In the next section and Section $3$, we give an approach to find divisibility of two type of sums of powers of consecutive integers, respectively. In Section $4$, we present another approach to study the divisibility of the powers of integers in terms of arithmetic series by using the Faulhaber's theorem and the formula shown in Chen, Fu, and Zhang \cite{CFZ}. Finally, we discuss the power sums expressed in terms of Stirling numbers of the second kind in Section $5$.

\section{Divisibility of sums of powers of consecutive integers}

It is well known that the sum of cubes  of three consecutive integers is always divisible by $9$ (see, for example, Rosen \cite{Ros}). Does it hold for other positive integer $k$ for which the sum of $k^{th}$ powers of $k$ consecutive integers is divisible for $k^2$? It is easy to check the answer is negative for $k=2,4$, etc. However, for odd $k$ the situation is quite different. In Ho, Mellblom, and Frodyma \cite{HMF}, it has been shown that for any odd positive integer $k$ and $m$, the sum of the $(mk)^{th}$ powers of consecutive integers 

\be\label{4.1}
n^{mk}+(n+1)^{mk}+\cdots+ (n+k-1)^{mk}
\ee
of any $k$ consecutive integers is always divisible by $k^2$. We will show this result can be proved by using the Girard-Waring identity. More precisely, let us denote by $S(m,k,n)$ the above sum $\pmod {k^2}$ of $(mk)^{th}$ powers of $k$ consecutive terms beginning with the integer $n$. We may drop the symbol $n$ and write $S(m,k)$ instead, namely,

\be\label{0-0-0}
S(m,k):=1^{mk}+2^{mk}+\cdots+k^{mk}.
\ee
 
\begin{proposition}\label{pro:4.1}
Let $x,y\in{\bN}$, and let $p\geq 3$ be a prime number. If $x+y=p$, then $x^p+y^p\equiv 0\,(mod\, p^2)$. Furthermore, for any odd positive integer $m$, there holds

\be\label{4.2}
x^{pm}+y^{pm}\equiv 0\, (mod\, p^2).
\ee
\end{proposition}

\begin{proof}
Substituting $x+y=z=p$ into Proposition \ref{pro:3.3}, we immediately get $p^2xy|(x^p+y^p-p^p)$, 
which implies $x^p+y^p\equiv 0\,(mod\, p^2)$, or equivalently, $x^p\equiv -y^p\, (mod\, p^2)$. If $m$ is an odd positive integer, then 

\[
(x^p)^m\equiv (-y^p)^m\, (mod\, p^2).
\]
Hence, we obtain \eqref{4.2}.
\end{proof}\eop

\begin{corollary}\label{cor:4.2}
Let $p\geq 3$ be a prime number, and let $m$ be an odd positive integer. Denote 

\[
S(m,p):=1^{mp}+2^{mp}+\cdots+p^{mp}
\]
and 

\[
S'(m,p):=1^{mp}+2^{mp}+\cdots+(p-1)^{mp}.
\]
Then $S(m,p)$ and $S'(m,p)\equiv 0\,(mod\,p^2)$.
\end{corollary}

\begin{proof}
From Proposition \ref{pro:4.1} for $k=1,2,\ldots, [p/2]$ 

\[
k^{mp}+(p-k)^{mp}\equiv 0\, (mod\, p^2).
\]
Hence 

\[
S(m,p)=\sum^{[p/2]}_{k=1}(k^{mp}+(p-k)^{mp})+p^{mp}\equiv 0\, (mod\, p^2).
\]
$S'(m,p)\equiv 0\,(mod\,p^2)$ can be obtained because $p^{mp}\equiv 0\, (mod\, p^2)$ ($p\geq 3$). 
\end{proof}\eop

\begin{remark}\label{rem:4.1} The result on $S(m,p)$ of Corollary \ref{cor:4.2} is included in Theorem $2$ of \cite{HMF} by using  a different approach.
\end{remark}

A similar process can be applied to extend the above result.

\begin{proposition}\label{pro:4.3}
Let $p\geq 3$ be a prime number, and let $m$ be an odd positive integer. For any $t\in{\bN}$, denote 

\[
\hat S(m,p^t):=1^{mp}+2^{mp}+\cdots+(p^t)^{mp}
\]
and 

\[
\hat S'(m,p^t):=1^{mp}+2^{mp}+\cdots+(p^t-1)^{mp}.
\]

Then $\hat S(m,p^t)$ and $\hat S'(m,p^t)\equiv 0\,(mod\,p^{t+1})$.
\end{proposition}

\begin{proof}
Let $x,y\in {\bN}$ with $x+y=p^t$, and let $p\geq 3$ be a prime number. Then from Proposition \ref{pro:3.3}, 

\[
p^{t+1}xy|(x^p+y^p-(p^t)^p),
\]
which implies $x^p+y^p\equiv 0\,(mod\, p^{t+1})$. If $m$ is an odd positive integer, then 

\[
x^{pm}+y^{pm}\equiv 0\, (mod\, p^{t+1})
\]
when $x+y=p^t$. Thus 

\bns
&&\hat S(m,p^t)=1^{mp}+2^{mp}+\cdots +(p^t-2)^{pm}+(p^t-1)^{pm}+(p^t)^{mp}\\
&=&(p^t)^{mp}+(1^{mp}+(p^t-1)^{mp})+(2^{mp}+(p^t-2)^{pm})+\cdots \equiv 0\,(mod \, p^{t+1}).
\ens
$\hat S'(m,p^t)\equiv 0\,(mod\,p^{t+1})$ follows. Here, $\hat S(m, p^t)\equiv 0\, (mod\, p^{t+1})$ because $(p^t)^{mp}\equiv 0\,(mod\, p^{t+1})$. 
\end{proof}\eop

\begin{remark}\label{rem:4.2} Formula (4) of Hsu \cite{Hsu} presents 
\[
\sum^n_{k=1}k^m=\sum^m_{j=1}j!\stirl{m}{j}\binom{n+1}{j+1},
\]
where $\stirl {m}{j}$ are the Stirling numbers of the second kind. Thus for an odd integer $m$ 
\[
1^{mp}+2^{mp}+\cdots+(p^t)^{mp}=\sum^{mp}_{j=1}j!\stirl{mp}{j}\binom{p^t+1}{j+1}.
\]
Then, from Proposition \ref{pro:4.3}, we have 
\[
\sum^{mp}_{j=1}j!\stirl{mp}{j}\binom{p^t+1}{j+1}\equiv 0\, (mod\, p^{t+1})
\]
for any prime $p\geq 3$.
\end{remark}

We now consider the general case of $p=k\in{\bN}$ is any odd positive integer. 

\begin{proposition}\label{pro:4.4}
Let $x,y\in{\bN}$ with $x+y=k$, an odd positive integer. Then 

\be\label{4.3}
x^k+y^k\equiv 0\,(mod\, k^2).
\ee
\end{proposition}

\begin{proof}
Note $[k/2]=(k-1)/2$ when $k\in{\bN}$ is an odd number, and for $1\leq i\leq (k-1)/2$ $\frac{k}{k-i}\binom{k-i}{i}$ are integers, since they are the coefficients of the Lucas polynomial. It can also be proved from the following observation:

\bns
&&\frac{k}{k-i} \binom{k-i}{i} =\binom{k-i}{i}+\binom{k-i-1}{i-1}\\
&&=2\binom{k-i}{i}-\binom{k-i-1}{i}.
\ens
Then, from the Girard-Waring identity \eqref{eq:1.1} we have 

\bns
&&x^k+y^k-(x+y)^k\\
&=&\sum_{1\leq i\leq (k-1)/2}(-1)^i\frac{k}{k-i}\binom{k-i}{i}(x+y)^{k-2i}(xy)^i\\
&=&\sum_{1\leq i\leq (k-3)/2}(-1)^i\frac{k}{k-i}\binom{k-i}{i}(x+y)^{k-2i}(xy)^i\\
&&\quad +(-1)^{(k-1)/2}\frac{2k}{k+1}\binom{\frac{k+1}{2}}{\frac{k-1}{2}}(x+y)(xy)^{(k-1)/2}\\
&=&\sum_{1\leq i\leq (k-3)/2}(-1)^i\frac{k}{k-i}\binom{k-i}{i}k^{k-2i}(xy)^i +(-1)^{(k-1)/2}k^2(xy)^{(k-1)/2},
\ens
where every term of the sum on the rightmost contains the factor $k^2$ because that the exponent of $k^{k-2i}$, $k-2i\geq 3$. Note the coefficient of the last term comes from

\[
(-1)^{(k-1)/2}\frac{2}{k+1}\binom{\frac{k+1}{2}}{\frac{k-1}{2}}=(-1)^{(k-1)/2}\frac{2}{k+1}\frac{k+1}{2}=(-1)^{(k-1)/2}.
\]
Thus, we obtain 

\[
x^k+y^k-k^k\equiv 0\,\pmod{k^2},
\]
which implies \eqref{4.3}. 
\end{proof}\eop

\begin{proposition}\label{pro:4.5}
Let $k$ be an odd positive integer, and let $m$ be an odd positive integer. Denote 

\[
S(m,k):=1^{mk}+2^{mk}+\cdots +k^{mk}
\]
and 

\[
S'(m,k):=1^{mk}+2^{mk}+\cdots +(k-1)^{mk}.
\]
Then 

\be\label{4.4}
S(m,k), S'(m,k)\equiv 0\,(mod\, k^2).
\ee
\end{proposition}

\begin{proof}
From Proposition \ref{pro:4.4}, we have $x^k\equiv -y^k\,(mod\, k^2)$. For an odd positive integer 
$m$, $(x^k)^m\equiv -(y^k)^m\, (mod\, k^2)$ follows. Thus $x^{km}+y^{km}\equiv 0\, (mod\, k^2)$ 
when $x+y=k$. Consequently, we have \eqref{4.4}.
\end{proof}\eop

\begin{remark}\label{rem:4.3} The result on $S(m,k)$ of Proposition \ref{pro:4.5} is included in Theorem $1$ of \cite{HMF} by using  a different approach.
\end{remark}

Considering the case of odd $k$ and even $m$, we may have the following result.

\begin{proposition}\label{pro:4.6}
Let $k$ be an odd positive integer, and let $m$ be an even positive integer. Then

\bn\label{4.5}
&&S(m,k):=1^{mk}+2^{mk}+\cdots +k^{mk}\nonumber\\
&\equiv& 2 (-1)^{m/2}\sum_{u+v=k, u<v}(uv)^{mk/2}\, (mod\, k^2).
\en
\end{proposition}

\begin{proof}
From Proposition \ref{pro:4.4}, we have $x^k+y^k\equiv 0\, (mod\, k^2)$ if $x+y=k$. Denote $a=x^k$ and $b=y^k$. From \eqref{eq:1.1}, for an even positive 
integer $m$ 

\bns
&&a^m+b^m=\sum_{0\leq i\leq m/2}(-1)^i\frac{m}{m-i}\binom{m-i}{i}(a+b)^{m-2i}(ab)^i\\
&=&\frac{m}{m}\binom{m}{0}(a+b)^m+(-1)\frac{m}{m-1}\binom{m-1}{1}(a+b)^{m-2}(ab)^1+\cdots\\
&&\quad  +(-1)^{(m/2)-1}\frac{m}{m-(m/2)+1}\binom{\frac{m}{2}+1}{\frac{m}{2}-1}(a+b)^2(ab)^{(m/2)-1}\\
&&\quad +(-1)^{m/2}\frac{m}{m/2}\binom{\frac{m}{2}}{\frac{m}{2}}(ab)^{m/2}.
\ens
On the rightmost side of the above equation, all terms except the last one contain factor $k^2$ because $a+b=x^k+y^k$. Thus, 

\[
a^m+b^m=x^{mk}+y^{mk}\equiv 2(-1)^{m/2}(xy)^{mk/2}\, (mod\, k^2).
\]
Consequently, 

\bns
&&S(m,k):=1^{mk}+2^{mk}+\cdots +k^{mk}\\
&=&(1^{km}+(k-1)^{km})+(2^{km}+(k-2)^{km})+\cdots\\
&&\quad  +\left( \left[\frac{k}{2}\right]^{km}+\left(\left[\frac{k}{2}\right]+1\right)^{km}\right)+k^{km}\\
&\equiv& 2(-1)^{m/2}\sum_{u+v=k,u< v}(uv)^{mk/2}\, (mod\, k^2).
\ens
\end{proof}\eop

\begin{example}\label{ex:3.1} 
As examples of Proposition \ref{pro:4.6}, we have 

\bns
&&S(2,3)=(1^3)^2+(2^3)^2+(3^3)^2\equiv 2(-1)\sum_{u+v=3, 0<u<v}(uv)^3\,(mod\, 3^2)\\
&&\equiv -2(1\cdot 2)^3\, (mod\, 3^2) \equiv 2\, (mod\, 3^2)\\
&&S(4,3)=(1^3)^4+(2^3)^4+(3^3)^4\equiv 2(-1)^2\sum_{u+v=3, 0<u<v}(uv)^6\,(mod\, 3^2)\\
&&\equiv 2(1\cdot 2)^6\,(mod\, 3^2)\equiv 2\, (mod\, 3^2).
\ens
\end{example}

In general, we have the following result for sequence $(S(2\ell,3))_{\ell \geq 1}$.

\begin{corollary}\label{cor:4.7}
For $\ell=1,2,\ldots$, 

\be\label{4.6}
S(2\ell,3)\equiv 2\, (mod\, 3^2).
\ee
\end{corollary}

\begin{proof}
Example shows that  $S(2,3)\equiv 2\, (mod\, 3^2)$. Using mathematical induction, one may prove that \eqref{4.6} is true. 
However, a direct proof can be given without using induction. 

\begin{align*}
S(2\ell, 3)\equiv&2(-1)^\ell\sum_{u+v=3,0<u<v}(uv)^{3\ell}\, (mod\, 3^2)\\
=&2 (-1)^\ell2^{3\ell}\,(mod\, 3^2)\\
=&2(-1)^\ell (2^3)^\ell\,(mod\, 3^2)\\
=&2(-1)^\ell (-1)^\ell\,(mod\, 3^2)=2\,(mod\, 3^2).
\end{align*}
\end{proof}
\eop

\begin{example}\label{ex:3.2} 
This example gives a divisible case for odd $k$ and even $m$. From Proposition \ref{pro:4.6} we obtain 

\bns
&&S(2,7):=1^7+(2^7)^2+\cdots +(7^7)^2\equiv -2\sum_{u+v=7, u< v}(uv)^k\, (mod\, 7^2)\\
&\equiv&-2\left( (1\cdot 6)^7+(2\cdot 5)^7+(3\cdot 4)^7\right) \, (mod\, 7^2)\\
&\equiv& -92223488\,(mod\, 7^2)\equiv 0 \,(mod\, 7^2).
\ens
\end{example}

We now extend the result shown in Corollary \ref{cor:4.7} to $S(m,p)$ for a prime $p\geq 3$ and some $m$, which is a special case of $S(m,k)$ with odd integer $k$. In fact, \cite{HMF} gives the following result.

\begin{proposition}\label{pro:4.8}\cite{HMF}
Let $p\geq 3$ be a prime number, and let $\ell$ be a positive integer. Then 

\be\label{4.7}
S((p-1)\ell, p)\equiv (p-1)\, (mod\, p^2).
\ee
\end{proposition}

\begin{proof} We present a brief proof here for the convenience of the readers. We consider two cases. The first case is for $(a,p)=1$. Then, from Euler's phi function theorem, we have $a^{\phi (p^2)}=a^{p(p-1)}\equiv 1\, (mod\, p^2)$. Thus, for $1\leq i\leq p-1$, we have $(i^\ell)^{p(p-1)}\equiv 1 \, (mod\, p^2)$. The second case is for $p|a$. Then, $(p^\ell)^{p(p-1)}\equiv 0\,(mod\, p^2)$. Thus, we have 

\begin{align*}
S((p-1)\ell, p)=&(1^\ell)^{p(p-1)}+(2^\ell)^{p(p-1)}+\cdots +((p-1)^\ell)^{p(p-1)}+(p^\ell)^{p(p-1)}\\
\equiv &(p-1)\, (mod\, p^2).
\end{align*}
\end{proof}
\eop

\begin{corollary}\label{cor:4.9}
Let $p\geq 3$ be a prime number. Then 

\be\label{4.8}
2 \sum_{u+v=p,0<u<v}(-uv)^{p(p-1)/2}\equiv (p-1)\, (mod\, p^2).
\ee
\end{corollary}

\begin{proof}

From Proposition \ref{pro:4.6} we have 

\bn\label{4.9}
&&S(p-1,p):=1^{p(p-1)}+2^{p(p-1)}+\cdots +p^{p(p-1)}\nonumber\\
&\equiv& 2 (-1)^{(p-1)/2}\sum_{u+v=p,0<u<v}(uv)^{p(p-1)/2}\, (mod\, p^2),
\en
which implies \eqref{4.8} after using \eqref{4.7} for $\ell=1$.
\end{proof}\eop

\begin{remark}\label{rem:3.1} 
An alternative proof of \eqref{4.7} can be represented by considering two cases. The first case is for even $\ell$. Then, by using Proposition \ref{pro:4.6} and noting $p\geq 3$ is an odd positive integer, we have 

\bn\label{4.10}
&&S((p-1)\ell,p):=1^{p(p-1)\ell}+2^{p(p-1)\ell}+\cdots +p^{p(p-1)\ell}\nonumber\\
&\equiv& 2 (-1)^{(p-1)\ell/2}\sum_{u+v=p,0<u<v}(uv)^{p(p-1)\ell/2}\, (mod\, p^2).
\en
Thus, for all $0<u<v$ with $u+v=p$

\[
(uv)^{p(p-1)\ell/2}=\left((uv)^{\ell/2}\right)^{p(p-1)}\equiv 1\,(mod\, p^2)
\]
which yields 

\[
S((p-1)\ell,p)\equiv 2(-1)^{(p-1)\ell/2}\left[ \frac{p}{2}\right] \, (mod\, p^2)\equiv 2\frac{p-1}{2}\, (mod\, p^2).
\]

The second case is for odd $\ell$. Then, by noting $\ell=2k+1$ for a positive integer $k$ and $p\geq 3$ is an odd positive integer, from Proposition \ref{pro:4.6} we have 

\bns
&&2 (-1)^{(p-1)\ell/2}\sum_{u+v=p,0<u<v}(uv)^{p(p-1)\ell/2}\\
&=&2(-1)^{(p-1)/2}\sum_{u+v=p,0<u<v}(uv)^{p(p-1)k}(uv)^{p(p-1)/2}\\
&\equiv& 2(-1)^{(p-1)/2}\sum_{u+v=p,0<u<v}(uv)^{p(p-1)/2}\, (mod\, p^2).
\ens
Thus, the problem is reduced to prove 
\[
2(-1)^{(p-1)/2}\sum_{u+v=p,0<u<v}(uv)^{p(p-1)/2}\equiv (p-1)\, (mod\, p^2),
\]
completing the proof. A short cut process for the case of odd $\ell$ maybe done as follows 

\[
S((p-1)\ell,p)\equiv S(p-1,p)\,(mod\, p^2)\equiv (p-1)\, (mod\, p^2).
\]
\end{remark}

\section{Divisibility of the sum $1^k+2^k+\cdots +n^k$}

We now establish the following general results on the divisibility of the sum

\be\label{4.0}
S_k(n):=1^k+2^k+\cdots +n^k
\ee
for odd integers $k\geq 3$ and $n\in{\bN}$. 

\begin{proposition}\label{pro:4.10}
Let $S_k(n)$ be the sum defined by \eqref{4.0}, and let $n$ and $k\geq 3$ be odd integers satisfying $k\equiv 0\, (mod\, n)$. Then 

\be\label{4.-1}
S_k(n)\equiv 0\, (mod\, n^2).
\ee
\end{proposition}

\begin{proof}
Noting for odd integers $n$ and $k\geq 3$ 

\[
i^k+(n-i)^k\equiv i^k+kn(-i)^{k-1}+(-i)^k\equiv kni^{k-1}\, (mod\, n^2),
\]
we may write $S_k(n)$ as

\bn
&&S_k(n)=\left[ 1^k+(n-1)^k\right]+\left[ 2^k+(n-2)^k\right]+\cdots \nonumber\\
&&+\left[ \left(\frac{n-1}{2}\right)^k+\left( \frac{n+1}{2}\right)^k\right] +n^k\nonumber\\
&\equiv &kn\left[1^{k-1}+2^{k-1}+\cdots +\left( \frac{n-1}{2}\right)^{k-1}\right]+n^k\equiv 0\, (mod\, n^2),\label{4-2}
\en
where, in the last step, we used $k\equiv 0\, (mod\, n)$. This completes the proof. 
\end{proof}\eop

In Damianou and Schumer \cite{DS}, Von Staudt-Clausen Theorem (see Theorem 118 of Hardy and Wright \cite{HW}) is applied to establish the following result.

\begin{theorem}\label{thm:4.11} \cite{DS}
Let $S_k(n)$ be the sum defined by \eqref{4.0}. Then $S_k(n)\equiv 0\, (mod\, n)$ if and only if for every prime $p$ that divides $n$, $p-1\nmid k$.
\end{theorem}

By using Theorem \ref{thm:4.11} and noting the decomposition formula shown in \eqref{4-2} we obtain the following conditions of $n$ and $k$ for the divisibility $n^2|S_k(n)$.

\begin{theorem}\label{thm:4.12}
Let $S_k\left(p^\alpha\right)$ be the sum defined by \eqref{4.0}, and let $p$ be an odd prime, $\alpha \in{\bN}$, and $k\geq 3$ an odd integers. Then the divisibility $\left(p^\alpha\right)^2|S_k\left(p^\alpha\right)$ holds if $p-1\nmid (k-1)$.
\end{theorem}

\begin{proof}
From \eqref{4-2}, we have 

\[
S_k(n)\equiv kn\left[1^{k-1}+2^{k-1}+\cdots +\left( \frac{n-1}{2}\right)^{k-1}\right]+n^k\, (mod\, n^2).
\]
Since for $1\leq i\leq n-1$ and odd $k\geq 3$, $i^{k-1}\equiv (n-i)^{k-1}\, (mod\, n)$ we have 

\bns
&&1^{k-1}+2^{k-1}+\cdots +\left( \frac{n-1}{2}\right)^{k-1}\\
&\equiv& (n-1)^{k-1}+(n-2)^{k-1}+\cdots +\left( n-\frac{n-1}{2}\right)^{k-1}\, (mod\, n)\\
&=&(n-1)^{k-1}+(n-2)^{k-1}+\cdots +\left( \frac{n+1}{2}\right)^{k-1}\, (mod\, n).
\ens
Consequently, for $k\geq 3$ 

\[
S_{k-1}(n)\equiv S_{k-1}(n-1)\equiv 2 \left[1^{k-1}+2^{k-1}+\cdots +\left( \frac{n-1}{2}\right)^{k-1}\right]\, (mod\, n).
\]
Particularly, 

\[
S_{k-1}\left(p^\alpha\right)\equiv S_{k-1}\left(p^\alpha-1\right)\equiv 2 \left[1^{k-1}+2^{k-1}+\cdots +\left( \frac{p^\alpha-1}{2}\right)^{k-1}\right]\, \left(mod\, p^\alpha\right).
\]
Thus, from Theorem \ref{thm:4.11} we obtain that $p-1\nmid (k-1)$ implies 

\[
1^{k-1}+2^{k-1}+\cdots +\left( \frac{p^\alpha-1}{2}\right)^{k-1}\equiv 0\, \left(mod\, p^\alpha\right),
\]
so that 

\[
S_k(p^\alpha)\equiv k p^\alpha \left[ 1^{k-1}+2^{k-1}+\cdots +\left( \frac{p^\alpha-1}{2}\right)^{k-1}\right]+(p^\alpha)^k\equiv 0\, \left(mod\, (p^\alpha)^2\right).
\]
\end{proof}\eop

\begin{theorem}\label{thm:4.13}
Let $S_k\left(pq\right)$ be the sum defined by \eqref{4.0}, and let $p$ and $q$ be distinct odd primes and $k\geq 3$ an odd integers. Then the divisibility $\left(pq\right)^2|S_k\left(pq\right)$ holds if $k\in A\backslash B$, where 
\[
A=\{k:k\in{\bN}\, and\, k\not\equiv 1\, (mod\, p-1)\}\quad \mbox{and}\quad B=\{ k:k\in{\bN}\, and\, k\equiv 1\, (mod\, q-1)\}.
\]
Furthermore, the divisibility $\left(pq\right)^2|d\cdot S_k\left(pq\right)$ holds if $k\in A$, where 

\[
d=\left\{ \begin{array}{ll} q& if \, k\equiv 1\,(mod\, q-1)\, and \, k\not=q\\1 &otherwise\end{array}.\right.
\]
\end{theorem}

\begin{proof}
We now prove the first part of the theorem. Since $k$ is odd, in view of \eqref{4-2}, we have 

\be\label{4-2-2}
S_k(n)\equiv kn\left( 1^{k-1}+2^{k-1}+\cdots+\left( \frac{n-1}{2}\right)^{k-1}\right)+n^k\,(mod\, n^2).
\ee
Owing to $p-1\nmid k-1$ and $q-1\nmid k-1$, by Theorem \ref{thm:4.11}, we have 

\be\label{4-2-3}
S_{k-1}(pq)\equiv 0\, (mod\, pq).
\ee
Since $k$ is odd, we get $i^{k-1}\equiv (n-i)^{k-1}\,(mod\, n)$. It follows that 

\begin{align}
S_{k-1}(pq)=& 1^{k-1}+2^{k-1}+\cdots+\left( \frac{pq-1}{2}\right)^{k-1}\nonumber\\
+& \left( \frac{pq+1}{2}\right)^{k-1}+\cdots +(pq-1)^{k-1}+(pq)^{k-1}\nonumber\\
\equiv& 2\left( 1^{k-1}+2^{k-1}+\cdots +\left( \frac{pq-1}{2}\right)^{k-1}\right)\, (mod\, pq).\label{4-2-4}
\end{align}
Combining \eqref{4-2-3} and \eqref{4-2-4} yields that 

\[
2\left( 1^{k-1}+2^{k-1}+\cdots +\left( \frac{pq-1}{2}\right)^{k-1}\right)\equiv 0\, (mod\, pq).
\]
Under the condition that $p$ and $q$ are odd primes, we see that 

\[
\left( 1^{k-1}+2^{k-1}+\cdots +\left( \frac{pq-1}{2}\right)^{k-1}\right)\equiv 0\, (mod\, pq).
\]
which, together with \eqref{4-2-2}, implies that 

\[
S_k(pq)\equiv 0\, (mod\, (pq)^2).
\]

To prove the second part of the theorem, we use \eqref{4-2-2} to get 

\[
qS_k(n)\equiv knq\left( 1^{k-1}+2^{k-1}+\cdots +\left( \frac{n-1}{2}\right)^{k-1}\right)+qn^k\, (mod\, n^2).
\]
Since $k\geq 3$, we get 

\be\label{4-2-5}
qS_k(pq)\equiv kpq^2\left( 1^{k-1}+2^{k-1}+\cdots +\left( \frac{pq-1}{2}\right)^{k-1}\right)\, (mod\, (pq)^2).
\ee
Next we show that if $p-1\nmid k-1$, then 

\[
p| \left( 1^{k-1}+2^{k-1}+\cdots +\left( \frac{pq-1}{2}\right)^{k-1}\right).
\]
Since 

\[
\frac{pq-1}{2}=p\cdot \frac{q-1}{2}+\frac{p-1}{2},
\]
we have 

\begin{align}
&1^{k-1}+2^{k-1}+\cdots +\left( \frac{pq-1}{2}\right)^{k-1}\nonumber\\
=& (1^{k-1}+2^{k-1}+\cdots+p^{k-1})+((p+1)^{k-1}+(p+2)^{k-1}+\cdots +(2p)^{k-1})\nonumber\\
&+\cdots +\left( \left( p\frac{q-3}{2}+1\right)^{k-1}+\left( p\frac{q-3}{2}+2\right)^{k-1}+\cdots 
+\left( p\frac{q-3}{2}+p\right)^{k-1}\right)\nonumber\\
&+\left( \left( p\frac{q-1}{2}+1\right)^{k-1}+\left( p\frac{q-1}{2}+2\right)^{k-1}+\cdots 
+\left( p\frac{q-1}{2}+\frac{p-1}{2}\right)^{k-1}\right)\nonumber\\
\equiv&\frac{q-1}{2}(1^{k-1}+2^{k-1}+\cdots+p^{k-1})+\left( 1^{k-1}+2^{k-1}+\cdots +\left( \frac{p-1}{2}\right)^{k-1}\right)\, (mod\, p).\label{4-2-6}
\end{align}
Recalling that $p-1\nmid k-1$, we obtain 

\be\label{4-2-7}
S_{k-1}(p)\equiv 0\, (mod\, p).
\ee
The condition that $k\geq 3$ is odd implies $i^{k-1}\equiv (p-i)^{k-1}\,(mod\, p)$. Thus,

\begin{align}\label{4-2-8}
S_{k-1}(p)=&1^{k-1}+2^{k-1}+\cdots +\left( \frac{p-1}{2}\right)^{k-1}\nonumber\\
&+\left(\frac{p+1}{2}\right)^{k-1}+\cdots +(p-1)^{k-1}+p^{k-1}\nonumber\\
=& 2\left( 1^{k-1}+2^{k-1}+\cdots +\left( \frac{p-1}{2}\right)^{k-1}\right)\,(mod\, p).
\end{align}
Combining \eqref{4-2-7} and \eqref{4-2-8} and noting that $p$ is an odd prime, we obtain 
that 

\be\label{4-2-9}
\left( 1^{k-1}+2^{k-1}+\cdots +\left( \frac{p-1}{2}\right)^{k-1}\right)\equiv 0\, (mod\, p).
\ee
Therefore, if $k\equiv 1\,(mod\, q-1)$ and $k\not= q$, by \eqref{4-2-5}, we have 

\[
qS_k(pq)\equiv 0\, (mod\, (pq)^2).
\]
If $k=q$, \eqref{4-2-2} and \eqref{4-2-9} imply that 

\[
S_k(pq)\equiv 0\, (mod\, (pq)^2).
\]
This completes the proof.
\end{proof}
\eop

\begin{example}\label{eq:2.1} 
For $p=5$ and $q=7$, $p-1\nmid (k-1)$ or $k\not\equiv 1\,(mod\, 4)$ implies $k\in \{ 4\ell+3: \ell\in{\bN}\}$. From Theorem \ref{thm:4.13}, $S_k (35)\equiv 0\, (mod\, 35^2)$ for all $k\in\{ 4\ell+3:\ell\in {\bN}\}\backslash \{6\ell+1:\ell\in{\bN}\}$. For instance, $k=3,11,15,23,27,35,39,$ etc. Furthermore, $dS_k (35)\equiv 0\, (mod\, 35^2)$ for all $k\in\{ 4\ell+3:\ell\in {\bN}\}$, where 

\[
d=\left\{ \begin{array}{ll} 7& if \, k\equiv 1\,(mod\, 6)\, and\, k\not=7\\1 &otherwise\end{array}.\right.
\] 
Thus $dS_k (35)\equiv 0\, (mod\, 35^2)$ for $k=3,7,11,15,19,23,27,31,35,39,$ etc.

Similarly, $dS_k (55)\equiv 0\, (mod\, 55^2)$, where 

\[
d=\left\{ \begin{array}{ll} 11& if \, k\equiv 1\,(mod\, 10)\, and\, k\not=11\\1 &otherwise\end{array},\right.
\] 
for $k=3,7,11,15,19,23,27,31,35,39,43,47,51,55,$ etc. $S_k(55)\equiv 0\, (mod\, 55^2)$ for $k=3,7,11,15,19,23,27,35,39,43,47,55,$ etc. 
\end{example}

\begin{example}\label{eq:2.2} 
$S_k(6n+1)\equiv 0\, (mod\, (6n+1)^2)$ for all $0\leq n\leq 20$ and the following $k$:

\[
3,5,7,11,15,23,27,35,39, 47,59,63,75,83,87,95,99,105,107,119, 123, etc.
\]
If the $6n+1$ is replaced by $30n+1$, then we may add $9,17,29, 53,$ and $57$ in the above list.
\end{example}

\section{Faulhaber's approach for the divisibility of sums of powers of integers}
Faulhaber's formula, named after Johann Faulhaber, expresses the sum of the $k$-th powers of the first $n$ positive integers $S_k(n)$ as a $(k+1)$th degree polynomial function of $n$, the coefficients involving Bernoulli numbers $B_j$, in the form

\be\label{4-3-0}
S_k(n)=\sum^n_{\ell=1}\ell^k=\frac{1}{k+1}\sum^k_{\ell =0}\binom{k+1}{\ell}B_\ell n^{k+1-\ell},
\ee
where we use the Bernoulli number of the second kind $B_1=1/2$. If we use the Bernoulli number of the first kind $B_1=-1/2$ and noting all Bernoulli numbers with odd index $k\geq 3$ are zero, then formula \eqref{4-3-0} becomes \eqref{0.0}, i.e., 

\be\label{4-3-0-2}
S_k(n)=\sum^n_{\ell=1}\ell^k=\frac{1}{k+1}\sum^k_{\ell =0}(-1)^\ell B_\ell \binom{k+1}{\ell }n^{k+1-\ell},
\ee
where $B_1=-1/2$.

We first give a corollary of Proposition \ref{pro:4.10} and Theorem \ref{thm:4.12}.

\begin{corollary}\label{pro:4.13}
Let $S_k(n)$ be the sum defined by \eqref{4.0} with odd integers $n$ and $k\geq 3$. Then if either (i) $k\equiv 0\, (mod\, n)$ or (ii) $n=p^\alpha$, a positive integer power of a prime $p$, and $p-1\not |(k-1)$ implies 

\be\label{4-3}
\frac{1}{k+1}\sum^k_{\ell =0}\binom{k+1}{\ell }B_\ell n^{k-\ell-1} \in{\bN},
\ee
where $B_1=1/2$, and 

\be\label{4-3-2}
\frac{1}{k+1}\sum^k_{\ell =0}(-1)^\ell B_\ell \binom{k+1}{\ell }n^{k-\ell -1}\in{\bN},
\ee
where $B_1=-1/2$.
\end{corollary}

\begin{proof} In case (i), by using Proposition \ref{pro:4.10} we know $\left. n^2\right | S_k(n)$. Thus, from Faulhaber's formulas \eqref{4-3-0} and \eqref{4-3-0-2}, we obtain \eqref{4-3} and \eqref{4-3-2}.  

In case (ii), by using Theorem \ref{thm:4.12} we know $\left. (p^\alpha)^2\right |S_k(p^\alpha)$. Thus, from 
Faulhaber's formulas \eqref{4-3-0} and \eqref{4-3-0-2}, we get \eqref{4-3} and \eqref{4-3-2}. 
\end{proof}

We now consider the divisibility of the sums of powers of arithmetic sequence $\{a+(i-1)d\}_{1\leq i\leq k}$, where $(d,k)=1$, denoted by $S(a,d;k)$ and defined by 

\be\label{3.1}
S(a,d;k):=\sum^k_{i=1}(a+(i-1)d)^k=a^k+(a+d)^k+\cdots +(a+(k-1)d)^k.
\ee

\begin{theorem}\label{thm:3.1}
Let $a,d,$ and $k\in{\bN}$, where $k$ is an odd number with $(d,k)=1$, and let $S(a,d;k)$ be the series defined by \eqref{3.1}. Then 
$S(a,d;k)\equiv 0\,(mod\, k^2)$.
\end{theorem}

\begin{proof}
For $i=1,2,\ldots, k$ we have 

\be\label{3.2}
(a+(i-1)d)\equiv j_i\,(mod\, k),
\ee
where $j_i\in\{ 1,2, \ldots, k\}$. Since $(d,k)=1$, $j_1, j_2,\ldots,$ and $j_k$ are distinct, otherwise $j_i=j_m$ for some $i> m$ implies 
\[
(a+(i-1)d)-(a+(m-1)d)=(i-m)d\equiv 0\, (mod\, k),
\]
From $(i-m)d\equiv 0\, (mod\, k)$, we have $k|d(i-m)$. Under the condition $(d,k)=1$, we infer that $k|(i-m)$, which is contrary to the fact that $k\nmid (i-m)$ for any $1\leq m<i\leq k$. Therefore, $j_1, j_2,\ldots,$ and $j_k$ are distinct. 

From \eqref{3.2}, for some $\ell$ we have 

\bns
&&(a+(i-1)d)^k-j_i^k=(j_i+\ell k)^k-j_i^k\\
&=&j_i^{k-1}\binom{k}{1}(\ell k )+j_i^{k-2}\binom{k}{2}(\ell k)^2+\cdots+\binom{k}{k}(\ell k)^k\\
&\equiv &0\, (mod\, k^2).
\ens
Thus

\[
\sum^k_{i=1}(a+(i-1)d)^k\equiv \sum^k_{i=1}j_i^k=\sum^k_{j=1}j^k\equiv 0\, (mod\, k^2),
\]
completing the proof of the theorem. 
\end{proof}
\eop

\begin{remark}\label{rem:3.2}
If $a=d=1$, then $S(1,1;k)=S(1,k)$ and the result of Theorem \ref{thm:3.1} reduces to Proposition \ref{pro:4.5}.
\end{remark}

Because of the relation $(a+bi)^m=b^m(a/b+i)^m$, there is no loss of generality to consider the sum of the powers of $x+i$, namely, $\sum^n_{i=1}(x+i)^m$. If $x$ is a positive integer, then the last sum can be written as  

\[
\sum^n_{i=1}(x+i)^m=\sum^{n+x}_{i=1}i^m-\sum^x_{i=1}i^m.
\]
By Faulhaber's theorem, the two sums on the right-hand side are polynomials in $(n+x)(n+x+1)$ and $x(x+1)$, respectively. Using the relation 

\[
(n+x)(n+x+1)=n(n+2x+1)+x(x+1),
\]
we see that 

\[
[(n+x)(n+x+1)]^i-[x(x+1)]^i=\sum^i_{k=1}\binom{i}{k}[n(n+2x+1)]^k[x(x+1)]^{i-k},
\]
which is a polynomial in $n(n+2x+1)$. Hence, we have the following result.

\begin{proposition}\label{pro:3.2}
Let $p\geq 3$ be a prime number, and let $n$ be a positive integer with $p$ as a factor. If $x$ is a positive integer satisfying $p|(2x+1)$, then 

\[
\sum^n_{i=1}(x+i)^m\equiv 0\, (mod\, p^2)
\]
for all $m=0,1,2,\ldots.$
\end{proposition}

\begin{proof}
It is sufficient to notice 

\[
p^2\left |\sum^i_{k=1}\binom{i}{k}[n(n+2x+1)]^k[x(x+1)]^{i-k}\right.
\]
\end{proof}
\eop

\section{Power sum and Stirling numbers of the second kind}

The Stirling numbers of the second kind, denoted by $\stirl{n}{k}$, count the number of ways to partition a set of $n$ labelled objects into $k$ nonempty unlabelled subsets. The Stirling numbers of the second kind may also be characterized as the coefficients of the expansion of  powers of an indeterminate $x$ in terms of the falling factorials $(x)_{n}=x(x-1)(x-2)\cdots (x-n+1).$ In particular, $(x)_0 = 1$ because it is an empty product. Hence, 

\be\label{5.1}
 \sum _{k=0}^{n}\stirl{n}{k}(x)_k=x^n.
\ee
Substituting $n\to n+1$ into \eqref{5.1} and noting $\stirl{n+1}{0}=0$ for all $n\geq 0$, we have 

\be\label{5.2}
\sum_{k=1}^{n+1}\stirl{n+1}{k}(x-1)_{k-1}=x^n.
\ee

Expression of $x^k$ shown in \eqref{5.1} may help us to write the power sum $S_k(n)=\sum^n_{j=1}j^k$ as 

\bns
S_k(n)&=&\sum^n_{j=1}\sum^k_{i=0}\stirl{k}{i}(j)_i\\
&=&\sum^n_{j=1} \sum^k_{i=0}\stirl{k}{i}i!\left( {j\atop i}\right)
\ens
Since $\stirl{k}{0}=0$ for all $k> 1$, by interchanging the sums of the rightmost side of the above equation and noting

\bns
\left( {n+1\atop i+1}\right)&=& \left( {n\atop i}\right)+\left( {n-1\atop i}\right)+\cdots +\left( {i+1\atop i}\right)+\left( {i+1\atop i+1}\right)\\
&=& \left( {n\atop i}\right)+\left( {n-1\atop i}\right)+\cdots +\left( {i+1\atop i}\right)+\left( {i\atop i}\right),
\ens
we obtain 

\bns
S_k(n)&=&\sum^k_{i=1}\stirl{k}{i}i!\left( \sum^n_{j=i}\left({j\atop i}\right)\right)=\sum^k_{i=1}\stirl{k}{i}i!\left({n+1\atop i+1}\right)\\
&=&\sum^k_{i=1}\stirl{k}{i}\frac{1}{i+1}(n+1)_{i+1}=(n+1)\sum^k_{i=1}\frac{1}{i+1}\stirl{k}{i}(n)_i.
\ens

Substituting \eqref{5.2} into $S_k(n)=\sum^n_{j=1}j^k$ and noting $(j-1)_{i-1}=(i-1)!\left( {j-1\atop i-1}\right)$ and 

\[
\left( {n\atop i}\right)=\left({n-1\atop i-1}\right)+\left( {n-2\atop i-1}\right)+\cdots +\left( {i-1\atop i-1}\right),
\]
we have 

\bns
S_k(n)&=&\sum^n_{j=1}\sum^{k+1}_{i=1}\stirl{k+1}{i}(j-1)_{i-1}\\
&=&\sum^n_{j=1} \sum^{k+1}_{i=1}\stirl{k+1}{i}(i-1)!\left( {j-1\atop i-1}\right)\\
&=&\sum^{k+1}_{i=1}\stirl{k+1}{i}\left( \sum^n_{j=i}(i-1)!\left( {j-1\atop i-1}\right)\right)\\
&=&\sum^{k+1}_{i=1}\stirl{k+1}{i}(i-1)! \left( {n\atop i}\right)\\
&=&\sum^{k+1}_{i=1}\frac{1}{i}\stirl{k+1}{i}(n)_i.
\ens
Hence, we obtain the following results. 

\begin{proposition}\label{pro:5.1}
Denote $S_k(n)=\sum^n_{j=1}j^k$ and let $\stirl{n}{k}$ be the Stirling numbers of the second kind. Then 

\bn\label{5.3}
&&S_k(n)=(n+1)\sum^k_{i=1}\frac{1}{i+1}\stirl{k}{i}(n)_i\\
&&S_k(n)=\sum^{k+1}_{i=1}\frac{1}{i}\stirl{k+1}{i}(n)_i.\label{5.4}
\en
\end{proposition}

\begin{remark}\label{rem:4.3} Formula \eqref{5.3} is familiar, for instance, see $(4)$ in \cite{Hsu}.
\end{remark}

Since an odd prime $p|\stirl{p}{i}$ for all $2\leq i\leq p-1$ and $\stirl{p}{1}=\stirl{p}{p}=1$, if $n=k=p$, from \eqref{5.3} we know that 

\[
\frac{n+1}{i+1}(n)_i=\frac{p+1}{i+1}i!\left({p\atop i}\right)
\]
is divisible by $p$ when $i=1$ and $i=p$. Hence we have the following corollary. 

\begin{corollary}\label{cro:5.2}
Denote $S_k(n)=\sum^n_{j=1}j^k$. Then $p|S_p(p)$, where $p\geq 3$ is a prime. 
\end{corollary}

\noindent{\bf Acknowledgements}

We sincerely thank Professor William Y. C. Chen for his twice carefully reading and helpful comments of about ten pages, particularly, his full proof of Theorem 3.4 and his improvements, modification and/or simplification of the proofs of Corollary 2.11, Proposition 2.13, and Remark 2.15, as well as his corrections of typos and grammar checking, all of those lead to a quietly improved and revised version of the original manuscript.

\end{document}